\documentclass[12pt,a4paper,reqno]{amsart}
\usepackage{graphicx}
\usepackage{enumerate}
\usepackage{amssymb}

\usepackage{hyperref}

\usepackage[utf8]{inputenc}
\usepackage[dvips]{color}
\usepackage{pst-all,pstricks,pst-plot,pstricks-add, pst-math, pst-xkey, pst-func}
\usepackage{pst-node,graphicx}
\usepackage{subfigure}

\usepackage{psfrag}

\newtheorem{theorem}{Theorem}
\newtheorem{lemma}[theorem]{Lemma}
\newtheorem{corollary}[theorem]{Corollary}
\newtheorem{proposition}[theorem]{Proposition}
\newtheorem{remark}[theorem]{Remark}

\newcommand{\R}{\mathbb{R}}

\DeclareMathOperator{\diver}{div}

\DeclareMathOperator{\res}{Res}

\definecolor{grisoscuro}{gray}{.70}
\definecolor{grismedio}{gray}{.85}
\definecolor{grisclaro}{gray}{1.0}

\author{A. Gasull}
\address{Departament de Matem\`{a}tiques \\
Universitat Aut\`{o}noma de Barcelona \\ Edifici C. 08193 Bellaterra, Barcelona. Spain}
\email{gasull@mat.uab.cat}

\author{H. Giacomini}
\address{Laboratoire de Math\'{e}matiques et Physique Th\'{e}orique.
 Facult\'{e} des
Sciences et Techniques. Universit\'{e} de Tours, C.N.R.S. UMR 7350. 37200 Tours.
France} \email{Hector.Giacomini@lmpt.univ-tours.fr}

\author{S. P\'erez-Gonz\'alez}
\address{Departament de Matem\`{a}tiques \\
Universitat Aut\`{o}noma de Barcelona \\ Edifici C. 08193 Bellaterra, Barcelona. Spain}
\email{setperez@mat.uab.cat}

\author{J. Torregrosa}
\address{Departament de Matem\`{a}tiques \\
Universitat Aut\`{o}noma de Barcelona \\ Edifici C. 08193 Bellaterra, Barcelona. Spain}
\email{torre@mat.uab.cat}

\thanks{The first, third and fourth authors are supported by the
 MINECO/FEDER grant number MTM2008-03437 and the Generalitat de Catalunya grant number
2009SGR410}



\date{}

\commby{}

\begin{document}

\title[A proof of Perko's conjectures for the B-T system]
{A proof of Perko's conjectures for the Bogdanov-Takens system}

\begin{abstract}
The Bogdanov-Takens system has at most one limit cycle and, in the
parameter space, it exists between a Hopf and a saddle-loop
bifurcation curves. The aim of this paper is to prove the Perko's
conjectures about some analytic properties of the saddle-loop
bifurcation curve. Moreover, we provide sharp piecewise algebraic
upper and lower bounds for this curve.
\end{abstract}

\maketitle

\section{Introduction}\label{se:1}

The Bogdanov-Takens system
\begin{equation*}
\left\{{\begin{array}{l}
x'=y,\\
y'=-n+by+x^2+xy,
\end{array}}\right.
\end{equation*}
has been introduced in~\cite{Bog1975,Tak1974,Tak1974b}. It provides a universal
unfolding of a cusp point of codimension 2 and it is considered in many basic text
books on bifurcation theory; see for instance~\cite{ChoLiWan94,GucHol83,Kuz98}.
Some global quantitative properties of its bifurcation diagram are not known. In
1992, Perko stated two conjectures about analytic properties of the saddle-loop
bifurcation curve in the parameter space; see \cite{Per92}. The aim of this work is
to prove both conjectures.

The interesting bifurcations only appear in the region $n>0,$
because in this case the system has two critical points, $(\pm
\sqrt{n},0)$, a saddle and a focus. Therefore it is natural to
introduce a new positive parameter $m=\sqrt{n}$. So, we will
consider the following expression of the above system
\begin{equation}\label{sys:bt_original}
\left\{{\begin{array}{l}
x'=y,\\
y'=-m^2+by+x^2+xy,\quad\mbox{with}\quad m>0.
\end{array}}\right.
\end{equation}

Before presenting our results, we recall the known properties about
the bifurcation diagram of system~\eqref{sys:bt_original}. All the
qualitative information of this diagram and part of the quantitative
one are known; see~\cite{LiRouWan90,Per92, Per, RouWag95}. In
particular, it is proved in~\cite{LiRouWan90} that this system has
at most one limit cycle and that when it exists it is hyperbolic and
unstable. This information, together with the fact that
system~\eqref{sys:bt_original} is a rotated family of vector fields
with respect to $b$, allow to show that the limit cycle exists if
and only if $b^*(m)<b< m $, for an unknown function $b^*(m)$. This
holds because fixing $m$ and decreasing $b$, a unique unstable limit
cycle borns via a Hopf bifurcation for $b= m $, increases
diminishing $b$, and disappears in a saddle-loop connection for
$b=b^*(m)$. The corresponding phase portraits are drawn in
Figure~\ref{fig:rotatorio} and a sketch of its bifurcation diagram
is given in Figure~\ref{fig:regiones}.

\begin{figure}[h]
\begin{center}
\begin{tabular}{cccc}
\psset{xunit=0.8cm,yunit=0.8cm,algebraic=true,arrowsize=5pt}
\begin{pspicture*}(-2.6,-1.5)(1,1.5)
\psdots[dotsize=3pt](0,0)(-1.5,0)
\psecurve[linewidth=0.5pt,plotpoints=500]{->}(1,-2)(0.5,-1)(0,-0.65)(-1.9,-0.9)
(-2.4,0)(-1.8,0.9)(-1,0.85)(-0.25,0.25)(0,0)
\psecurve[linewidth=0.5pt,plotpoints=500]{-<}(1,-2)(0.5,-1)(0,-0.65)(-1.9,-0.9)
(-2.4,0)(-1.8,0.9)(-1,0.85)(-0.25,0.25)(0,0)(0.35,-0.35)(0.5,-0.5)
\psecurve[linewidth=1.5pt,plotpoints=500]{-}(1,-2)(0.5,-1)(0,-0.65)(-1.9,-0.9)
(-2.4,0)(-1.8,0.9)(-1,0.85)(-0.25,0.25)(0,0)(0.35,-0.35)(0.5,-0.5)(1,-1)
\psecurve[linewidth=0.5pt,plotpoints=500]{-<}(-1.55,0)(-1.5,-0.05)(-1.4,0)(-1.5,0.15)
(-1.75,0)(-1.5,-0.35)(-0.95,0)(-1.5,0.6)(-2.15,0)(-1.9,-0.6)(-1.2,-0.7)(-0.25,-0.25)(0,0)
\psecurve[linewidth=0.5pt,plotpoints=500]{->}(-1.55,0)(-1.5,-0.05)(-1.4,0)(-1.5,0.15)
(-1.75,0)(-1.5,-0.35)(-0.95,0)(-1.5,0.6)(-2.15,0)(-1.9,-0.6)(-1.2,-0.7)(-0.25,-0.25)(0,0)
(0.35,0.35)(0.5,0.5)
\psecurve[linewidth=1.5pt,plotpoints=500]{-}(-1.55,0)(-1.5,-0.05)(-1.4,0)(-1.5,0.15)
(-1.75,0)(-1.5,-0.35)(-0.95,0)(-1.5,0.6)(-2.15,0)(-1.9,-0.6)(-1.2,-0.7)(-0.25,-0.25)
(0,0)(0.35,0.35)(0.5,0.5)(1,1)
\psecurve[linewidth=0.5pt,arrowsize=2pt,plotpoints=500]{->}(-0.5,0.5)(-0.25,0.4)(0,0.3)
(0.25,0.4)(0.5,0.5)
\psecurve[linewidth=0.5pt,arrowsize=2pt,plotpoints=500]{->}(-0.65,0.35)(-0.5,0.25)(-0.3,0)
(-0.5,-0.25)(-0.65,-0.35)
\psecurve[linewidth=0.5pt,arrowsize=2pt,plotpoints=500]{->}(0.65,-0.35)(0.5,-0.25)(0.3,0)
(0.5,0.25)(0.65,0.35)
\end{pspicture*}
& \psset{xunit=0.8cm,yunit=0.8cm,algebraic=true,arrowsize=5pt}
\begin{pspicture*}(-2.6,-1.5)(1,1.5)
\psdots[dotsize=3pt](0,0)(-1.5,0)
\psecurve[linewidth=0.5pt,plotpoints=500]{>-<}(0.5,-0.5)(0.35,-0.35)(0,0)(-0.35,0.35)(-0.5,0.5)
\psecurve[linewidth=0.5pt,plotpoints=500]{<->}(-0.5,-0.5)(-0.35,-0.35)(0,0)(0.35,0.35)(0.5,0.5)
\psecurve[linewidth=1.5pt,plotpoints=500]{-}(1,-1)(0.5,-0.5)(0.35,-0.35)(0,0)(-0.25,0.25)
(-1,0.75)(-1.8,0.75)(-2.3,0)(-1.8,-0.8)(-1,-0.75)(-0.35,-0.35)(0,0)(0.25,0.25)(0.5,0.5)(1,1)
\psecurve[linewidth=0.5pt,plotpoints=500,arrowsize=3pt]{-<}(-1.55,0)(-1.5,-0.05)(-1.4,0)
(-1.5,0.15)(-1.75,0)(-1.5,-0.35)(-0.95,0)(-1.5,0.6)(-2.15,0)(-1.9,-0.6)(-1.2,-0.7)(-0.45,-0.25)
\psecurve[linewidth=0.5pt,plotpoints=500]{-}(-1.55,0)(-1.5,-0.05)(-1.4,0)(-1.5,0.15)(-1.75,0)
(-1.5,-0.35)(-0.95,0)(-1.5,0.6)(-2.15,0)(-1.9,-0.6)(-1.2,-0.7)(-0.45,-0.25)(-0.25,0)
\psecurve[linewidth=0.5pt,arrowsize=2pt,plotpoints=500]{->}(-0.5,0.5)(-0.25,0.4)(0,0.3)
(0.25,0.4)(0.5,0.5)
\psecurve[linewidth=0.5pt,arrowsize=2pt,plotpoints=500]{<-}(-0.5,-0.5)(-0.25,-0.4)(0,-0.3)
(0.25,-0.4)(0.5,-0.5)
\psecurve[linewidth=0.5pt,arrowsize=2pt,plotpoints=500]{->}(0.65,-0.35)(0.5,-0.25)(0.3,0)
(0.5,0.25)(0.65,0.35)
\end{pspicture*}
& \psset{xunit=0.8cm,yunit=0.8cm,algebraic=true,arrowsize=5pt}
\begin{pspicture*}(-2.6,-1.5)(1,1.5)
\psdots[dotsize=3pt](0,0)(-1.5,0)
\psecurve[linewidth=0.5pt,plotpoints=500]{-<}(1,2)(0.5,1)(0,0.65)(-1.9,0.9)(-2.4,0)(-1.8,-0.9)
(-1,-0.85)(-0.25,-0.25)(0,0)
\psecurve[linewidth=0.5pt,plotpoints=500]{->}(1,2)(0.5,1)(0,0.65)(-1.9,0.9)(-2.4,0)(-1.8,-0.9)
(-1,-0.85)(-0.25,-0.25)(0,0)(0.35,0.35)(0.5,0.5)
\psecurve[linewidth=1.5pt,plotpoints=500]{-}(1,2)(0.5,1)(0,0.65)(-1.9,0.9)(-2.4,0)(-1.8,-0.9)
(-1,-0.85)(-0.25,-0.25)(0,0)(0.35,0.35)(0.5,0.5)(1,1)
\psecurve[linewidth=0.5pt,plotpoints=500]{->}(-0.55,0)(-1.15,-0.5)(-1.7,-0.5)(-2.2,0)
(-1.9,0.6)(-1.2 ,0.7)(-0.25,0.25)(0,0)
\psecurve[linewidth=0.5pt,plotpoints=500]{-<}(-0.55,0)(-1.15,-0.5)(-1.7,-0.5)(-2.2,0)
(-1.9,0.6)(-1.2,0.7)(-0.25,0.25)(0,0)(0.35,-0.35)(0.5,-0.5)
\psecurve[linewidth=1.5pt,plotpoints=500]{-}(-0.55,0)(-1.15,-0.5)(-1.7,-0.5)(-2.2,0)(-1.9,0.6)
(-1.2,0.7)(-0.25,0.25)(0,0)(0.35,-0.35)(0.5,-0.5)(1,-1)
\psecurve[linewidth=0.5pt,plotpoints=500,arrowsize=3pt]{<-}(-1.6,0)(-1.5,-0.1)(-1.3,0)(-1.5,
0.25)(-1.75,0)(-1.5,-0.25)(-1,-0.1)(-0.95,0)
\psccurve[linewidth=1.5pt,linecolor=blue,plotpoints=500](-1.5,0.4)(-0.65,0)
(-1.5,-0.4)(-2,0)
\psline[linewidth=1.5pt,linecolor=blue]{-<}(-1.499,0.4)(-1.5,0.4)
\psecurve[linewidth=0.5pt,arrowsize=2pt,plotpoints=500]{<-}(-0.5,-0.5)
(-0.25,-0.4)(0,-0.3)(0.25,-0.4)(0.5,-0.5)
\psecurve[linewidth=0.5pt,arrowsize=2pt,plotpoints=500]{->}(-0.65,0.35)(-0.5,0.25)(-0.3,0)
(-0.5,-0.25)(-0.65,-0.35)
\psecurve[linewidth=0.5pt,arrowsize=2pt,plotpoints=500]{->}(0.65,-0.35)(0.5,-0.25)(0.3,0)
(0.5,0.25)(0.65,0.35)
\end{pspicture*}
& \psset{xunit=0.8cm,yunit=0.8cm,algebraic=true,arrowsize=5pt}
\begin{pspicture*}(-2.6,-1.5)(1,1.5)
\psdots[dotsize=3pt](0,0)(-1.5,0)
\psecurve[linewidth=0.5pt,plotpoints=500]{-<}(1,2)(0.5,1)(0,0.65)(-1.9,0.9)(-2.4,0)(-1.8,-0.9)
(-1,-0.85)(-0.25,-0.25)(0,0)
\psecurve[linewidth=0.5pt,plotpoints=500]{->}(1,2)(0.5,1)(0,0.65)(-1.9,0.9)(-2.4,0)(-1.8,-0.9)
(-1,-0.85)(-0.25,-0.25)(0,0)(0.35,0.35)(0.5,0.5)
\psecurve[linewidth=1.5pt,plotpoints=500]{-}(1,2)(0.5,1)(0,0.65)(-1.9,0.9)(-2.4,0)(-1.8,-0.9)
(-1,-0.85)(-0.25,-0.25)(0,0)(0.35,0.35)(0.5,0.5)(1,1)
\psecurve[linewidth=0.5pt,plotpoints=500]{->}(-1.55,0)(-1.5,0.05)(-1.4,0)(-1.5,-0.15)(-1.75,0)
(-1.5,0.35)(-0.95,0)(-1.5,-0.6)(-2.15,0)(-1.9,0.6)(-1.2,0.7)(-0.25,0.25)(0,0)
\psecurve[linewidth=0.5pt,plotpoints=500]{-<}(-1.55,0)(-1.5,0.05)(-1.4,0)(-1.5,-0.15)(-1.75,0)
(-1.5,0.35)(-0.95,0)(-1.5,-0.6)(-2.15,0)(-1.9,0.6)(-1.2,0.7)(-0.25,0.25)(0,0)(0.35,-0.35)
(0.5,-0.5)
\psecurve[linewidth=1.5pt,plotpoints=500]{-}(-1.55,0)(-1.5,0.05)(-1.4,0)(-1.5,-0.15)(-1.75,0)
(-1.5,0.35)(-0.95,0)(-1.5,-0.6)(-2.15,0)(-1.9,0.6)(-1.2,0.7)(-0.25,0.25)(0,0)(0.35,-0.35)
(0.5,-0.5)(1,-1)
\psecurve[linewidth=0.5pt,arrowsize=2pt,plotpoints=500]{<-}(-0.5,-0.5)(-0.25,-0.4)(0,-0.3)
(0.25,-0.4)(0.5,-0.5)
\psecurve[linewidth=0.5pt,arrowsize=2pt,plotpoints=500]{->}(-0.65,0.35)(-0.5,0.25)(-0.3,0)
(-0.5,-0.25)(-0.65,-0.35)
\psecurve[linewidth=0.5pt,arrowsize=2pt,plotpoints=500]{->}(0.65,-0.35)(0.5,-0.25)(0.3,0)
(0.5,0.25)(0.65,0.35)
\end{pspicture*}\\
(i) $b<b^*(m)$ & (ii) $b=b^*(m)$ & (iii) $b^*(m)<b< m $ & (iv) $ b
\geq m$
\end{tabular}
\end{center}
\caption{\small Phase portraits of system~\eqref{sys:bt_original}}\label{fig:rotatorio}
\end{figure}
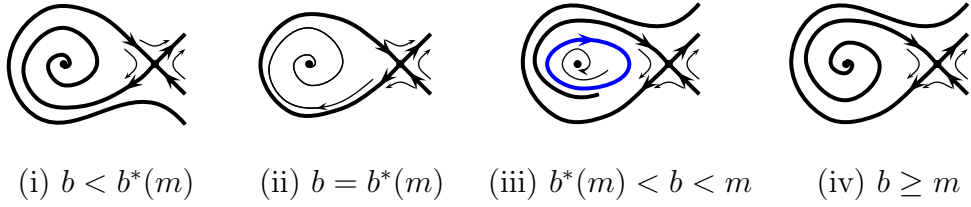

Some quantitative information about $b^*(m)$ is given by Perko in~\cite{Per92}:
\begin{enumerate}[(i)]
\item It is an analytic function.
\item It holds that $\max(-m,m-1)<b^*(m)< m .$
\item At $m=0$, $ b^*(m)=5 m/7+O(m^2). $ This term is computed by using the Melnikov method; see
also~\cite{GucHol83}.
\end{enumerate}

As usual, we write $f(m)=O(m^p)$ or $g(m)=o(m^p)$ at $m=m_0\in\R\cup\{\infty\}$ if
\[
\lim_{m\to m_0} \frac{f(m)}{m^p}=K\in\R,\quad \mbox{or}\quad \lim_{m\to m_0}\frac{g(m)}{m^p}=0.
\]

The lower bound given in item (ii) is improved in \cite{Hay04} applying the
Bendixson-Dulac Theorem and proving that $\max( m /2, m -1)<b^*(m)< m .$

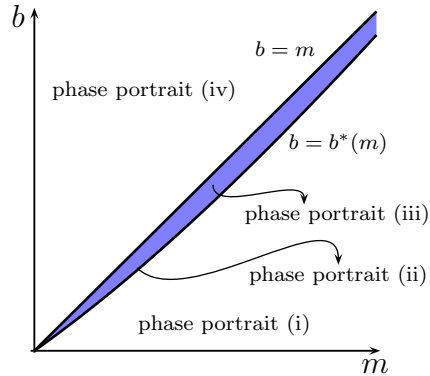
\begin{figure}[h]
\begin{center}
\psset{xunit=0.5cm,yunit=0.5cm,algebraic=true,arrowsize=3pt}
\begin{pspicture*}(-1,-0.6)(11,9.5)
\psline[linewidth=.7pt]{->}(-0.1,0)(9.1,0)
\psline[linewidth=.7pt]{->}(0,-0.05)(0,9.1)
\pscustom[fillstyle=solid,fillcolor=-green!50!red,linestyle=none]{
\psplot{0}{9}{(x)} \psplot{9}{0}{5*x/7+72*x^2/2401-30024*x^3/45294865}}
\psplot[linestyle=solid,linewidth=1pt,plotpoints=300]{0}{9}{(x)}
\psplot[linestyle=solid,linewidth=1pt,plotpoints=300]{0}{9}{5*x/7+72*x^2/2401-30024*x^3/45294865}
\rput(7.1,3.6){\rnode{C}{\phantom{xxxx}\tiny{phase portrait (iii)}}}
\rput(4.7,4.4){\pnode{D}}
\nccurve[linewidth=0.5pt,angleA=90,angleB=-90]{<-}{C}{D}
\rput(8,2){\rnode{E}{\tiny{phase portrait (ii)}}}
\rput(2.7,2.2){\pnode{F}}
\nccurve[linewidth=0.5pt,angleA=90,angleB=-45]{<-}{E}{F}
\put(4.3,-0.3){$m$} \put(-0.3,4.3){$b$}
\rput(2.9,6.9){\tiny{phase portrait (iv)}} \rput(5,0.7){\tiny{phase portrait (i)}}
\rput(6.6,8){\tiny{$b=m$}}
\rput(8,5.5){\tiny{$b=b^*(m)$}}
\end{pspicture*}
\end{center}
\caption{\small Sketch of the bifurcation diagram of system \eqref{sys:bt_original}. The open
colored region is the one containing the limit cycle}\label{fig:regiones}
\end{figure}

Item (iii) has been improved recently in~\cite{GasGiaTor10} using a
different approach, based on the construction of algebraic curves
with a loop that is without contact for the flow of the system. The
authors obtain that, at $m=0$,
\begin{align}\label{nonli}
b^*(m)&=\frac5 7 m+\frac{72}{2401}m^2 -\frac{30024}{45294865}m^3-
\frac{2352961656}{11108339166925} m^4+O(m^5).
\end{align}

The Bogdanov-Takens system for parameters in a neighborhood of infinity is also
studied in~\cite{Bou1991}. The aim of that work was to understand the presence of
the limit cycle for the system in terms of slow-fast dynamics. No quantitative
information about the shape of the curve $b=b^*(m)$ is given there.

In \cite{Per92} a different, but equivalent, expression of system
\eqref{sys:bt_original} is considered. Next conjectures correspond to Perko's ones
translated to \eqref{sys:bt_original}.

\newpage
\noindent{\bf Perko's Conjectures (\cite{Per92}):} Let $b=b^*(m)$ the function
that corresponds to the saddle-loop bifurcation curve for
system~\eqref{sys:bt_original}. Then,
\begin{itemize}
\item [($\mathcal{I}$)] for $m$ large enough, $b^*(m)= m -1+O(\frac 1{\sqrt{m}})$,
\item [($\mathcal{II}$)] it holds that $\max(5 m /7, m -1)<b^*(m).$
\end{itemize}

To facilitate the reading of this work, original Perko's formulation of above conjectures is
recalled in Section~\ref{ap:perko}.

Both conjectures are immediate consequences of Theorems~\ref{thm:n_grande} and 
\ref{thm:cota_global}. Moreover, Theorem \ref{thm:cota_global} significatively improves the global
lower and upper bounds given above.

\begin{theorem}\label{thm:n_grande}
For $m$ large enough, $b^*(m)= m -1+o\left(\frac1m\right).$
\end{theorem}

\begin{theorem}\label{thm:cota_global}
It holds that
\[
\max\left(\frac{5m}7,m-1\right)<b^*(m)<\min\left(\dfrac{(5+\frac{37}{12} m ) m
}{7+\frac{37}{12} m }, m -1+\dfrac{25}{7 m } \right).
\]
\end{theorem}

To prove our results we develop the method introduced in \cite{GasGiaTor10}, adapting it according
to small or large values of $m$. The basic idea is as follows: for given positive values of $b$ and
$m$ such that $b< m $, we want to know if $b>b^*(m)$ or $b<b^*(m)$; or equivalently to prove the
existence or non-existence of the limit cycle. Due to the uniqueness and hyperbolicity of the limit
cycle these two situations can be distinguished constructing negative or positively invariant
regions, as it is shown in Figure~\ref{fig:lazos}, and employing the Poincar\'{e}-Bendixson Theorem.

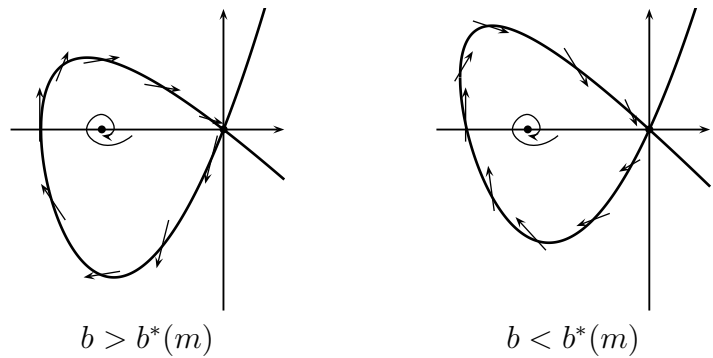
\begin{figure}[h]
\begin{center}
\begin{tabular}{c@{\hspace{2cm}}c}
\psset{xunit=0.8cm,yunit=0.8cm,algebraic=true,arrowsize=3pt}
\begin{pspicture*}(-3.5,-3)(1,2)
\psline[linewidth=.7pt]{->}(-3.5,0)(1,0)
\psline[linewidth=.7pt]{->}(0,-3)(0,2)
\psdots[dotsize=3pt](0,0)(-2,0)
\psplot[linestyle=solid,linewidth=1pt,plotpoints=1500]{-3}{1}{
2/9*1/(8256+141*sqrt(226)+266*x)*((11008*x+188*x*sqrt(226)+33024+564*sqrt(226)+sqrt(158811104*x^2 +
4645339*x^2*sqrt(226)+1844762403*x+57783633*x*sqrt(226)+4104987273+131542848*sqrt(226)))*x)}
\psplot[linestyle=solid,linewidth=1pt,plotpoints=1500]{-3}{1}{
2/9*1/(8256+141*sqrt(226)+266*x)*((11008*x+188*x*sqrt(226)+33024+564*sqrt(226)-sqrt(158811104*x^2
+ 4645339*x^2*sqrt(226)+1844762403*x+57783633*x*sqrt(226)+4104987273+131542848*sqrt(226)))*x)}
\psline[linestyle=solid,linewidth=0.5pt]{->}(-3.02,-0.2)(-3.02,0.7)
\psline[linestyle=solid,linewidth=0.5pt]{->}(-2.75,0.8)(-2.55,1.3)
\psline[linestyle=solid,linewidth=0.5pt]{->}(-2.3,1.1)(-1.7,1.2)
\psline[linestyle=solid,linewidth=0.5pt]{->}(-1.3,0.75)(-0.7,0.65)
\psline[linestyle=solid,linewidth=0.5pt]{->}(-0.4,0.22)(0,0.1)
\psline[linestyle=solid,linewidth=0.5pt]{->}(-2.6,-1.5)(-3,-0.9)
\psline[linestyle=solid,linewidth=0.5pt]{->}(-1.7,-2.35)(-2.3,-2.45)
\psline[linestyle=solid,linewidth=0.5pt]{->}(-0.9,-1.5)(-1.1,-2.3)
\psline[linestyle=solid,linewidth=0.5pt]{->}(-0.1,-0.1)(-0.3,-0.9)
\psecurve[linewidth=0.5pt,plotpoints=500,arrowsize=3pt]{<-}(-2.1,0)(-2.,-0.1)(-1.8,0)
(-2.,0.25)(-2.25,0)(-2.,-0.25)(-1.5,-0.1)(-1.45,0)
\end{pspicture*}
&
\psset{xunit=0.8cm,yunit=0.8cm,algebraic=true,arrowsize=3pt}
\begin{pspicture*}(-3.5,-3)(1,2)
\psline[linewidth=.7pt]{->}(-3.5,0)(1,0)
\psline[linewidth=.7pt]{->}(0,-3)(0,2)
\psdots[dotsize=3pt](0,0)(-2,0)
\psplot[linestyle=solid,linewidth=1pt,plotpoints=1500]{-3.110777}{1}{
(1/36)*(3549202749*x+403794459*x*sqrt(3121)+10770252205+834165035*sqrt(3121)
+sqrt(536203659176907980562*x^2+4006983259637287902*x^2*sqrt(3121)+5959548369623446683300*x+
47388386635309413180*x*sqrt(3121)+13496925785424777727250+106009783468206493150*sqrt(3121)))*x/
(468271835+36268045*sqrt(3121)+36402264*x)}
\psplot[linestyle=solid,linewidth=1pt,plotpoints=1500]{-3.110777}{1}{
(1/36)*(3549202749*x+403794459*x*sqrt(3121)+10770252205+834165035*sqrt(3121)
-sqrt(536203659176907980562*x^2+4006983259637287902*x^2*sqrt(3121) + 5959548369623446683300*x
+ 47388386635309413180*x*sqrt(3121)
+ 13496925785424777727250+106009783468206493150*sqrt(3121)))*x/(468271835
+ 36268045*sqrt(3121)+36402264*x)}
\psline[linestyle=solid,linewidth=0.5pt]{->}(-3.02,-0.2)(-3.02,0.7)
\psline[linestyle=solid,linewidth=0.5pt]{->}(-3.2,0.8)(-2.9,1.3)
\psline[linestyle=solid,linewidth=0.5pt]{->}(-2.9,1.8)(-2.3,1.6)
\psline[linestyle=solid,linewidth=0.5pt]{->}(-1.4,1.35)(-1,0.7)
\psline[linestyle=solid,linewidth=0.5pt]{->}(-0.4,0.5)(-0.2,0.05)
\psline[linestyle=solid,linewidth=0.5pt]{->}(-2.55,-1.35)(-2.65,-0.6)
\psline[linestyle=solid,linewidth=0.5pt]{->}(-1.7,-2)(-2.2,-1.45)
\psline[linestyle=solid,linewidth=0.5pt]{->}(-0.65,-1.39)(-1.2,-1.6)
\psline[linestyle=solid,linewidth=0.5pt]{->}(-0.15,-0.5)(-0.5,-0.7)
\psecurve[linewidth=0.5pt,plotpoints=500,arrowsize=3pt]{<-}(-2.1,0)(-2.,-0.1)(-1.8,0)
(-2.,0.25)(-2.25,0)(-2.,-0.25)(-1.5,-0.1)(-1.45,0)
\end{pspicture*}\\
$b>b^*(m)$&$b<b^*(m)$
\end{tabular}
\end{center}
\caption{\small Negative and positive invariant regions around the focus}\label{fig:lazos}
\end{figure}

The most difficult part of this approach consists in constructing these negative or positive
invariant regions delimited by loops. In \cite{GasGiaTor10}, the closed loop around the attracting
point to prove~\eqref{nonli} is proposed to be the loop of an algebraic self-intersecting curve
whose vertex is on the saddle point and whose branches approximate its separatrices. Here we use and
develop this idea when $m$ is small. One of the key points for proving that $b=5 m /7$ is a lower
bound of the saddle-node bifurcation curve is to consider a special rational parametrization of the
straight line $7b-5m=0,$ $m=m(s)$, $b=b(s)$; see \eqref{eq:parametrization}. With this
parametrization, the coordinates of both critical points on this line and the eigenvalues of the
saddle point are rational functions of $s.$ These facts diminish the computational difficulty
helping to prove the result.

When $m$ is large, we adapt the above approach constructing piecewise algebraic closed curves, which
also approximate the separatrices of the saddle point and takes into account the region where both
separatrices touch, for the first time, the negative $x$-axis.

The results that involve complicated and long algebraic manipulations are done both
with Mathematica and Maple.

The methods introduced in this work can also be useful to quantitatively study the unfolding of
other singularities. For instance, the cusp of codimension 3 (our case has codimension 2) considered
in~\cite{DumRouSot87}, the cases considered in~\cite{Dum77,DumFidLi01}, or the one studied by Takens
in~\cite{Tak1974b} and quoted in~\cite[p. 482]{Per} could be approached with our tools.

The paper is organized as follows. In Section~\ref{se:2}, changes of variables in
phase and parameter spaces are presented to shorten the
computations. The new results on the bifurcation curve, close to the
origin in the parameter space, are proved in Section~\ref{se:3} and
in Section~\ref{se:4} we prove the results for $m$ big enough.
Section~\ref{se:5} is devoted to prove Theorems~\ref{thm:n_grande}
and~\ref{thm:cota_global}. Finally, in Section~\ref{ap:perko},
Perko's formulation of the conjectures studied in this paper is
recalled and it is proved that they are equivalent to the ones
stated above.

\section{Changes of variables}\label{se:2}

As usual, to simplify the computations, we introduce changes of variables in both coordinates and
parameters. First, we move the saddle point of system~\eqref{sys:bt_original} to the origin. Second,
we change the parameters in such a way that the eigenvalues of the saddle point be rational
functions of the new parameters. We obtain in this way simpler expressions for the bifurcation
curves.

System~\eqref{sys:bt_original} is transformed by the change of variables $(x,y)\rightarrow (x+ m
,y)$ into
\begin{equation}\label{sys:bt_sqrt_nb}
\left\{{\begin{array}{l}
x'=y,\\
y'=2 m x+(b+m)y+x^2+xy.
\end{array}}\right.
\end{equation}
In order to obtain rational expressions for the eigenvalues
associated to the saddle point we introduce new parameters $M$ and
$B$ defined by
\begin{equation}\label{cv:nb_a_MB}
M^2=\frac{(b+ m )^2+8 m }{4},\quad B=\frac{b+ m }{2}.
\end{equation}
Then, system~\eqref{sys:bt_sqrt_nb} writes as
\begin{equation}\label{sys:bt_MB}
\left\{{\begin{array}{l}
x'=y,\\
y'=(M^2-B^2)x+2By+x^2+xy.
\end{array}}\right.
\end{equation}

The origin of~\eqref{sys:bt_MB} is a saddle point and the focus of \eqref{sys:bt_original},
$(-m,0)$, changes to $(B^2-M^2,0).$

In the parameter space the Hopf bifurcation curve $b= m $
becomes $B^2+2B-M^2=0.$ The condition of existence of an invariant
straight line $b= m -1$ is moved to $(B+1)^2-M^2=0$ and the lower
bound given in \cite{Hay04}, $b= m /2$, writes as $3B^2+8B-3M^2=0.$
Moreover, the origin goes to the origin and large values of $m$
also correspond with large values of $M.$

The homoclinic bifurcation curve $b=b^*(m)$ has a new expression
for system~\eqref{sys:bt_MB}. We denote\footnote{Although
numerically it seems that we can write it as $B=B^* (M)$ we have no
enough information on $b=b^*(m)$ to ensure this fact.} it by $W
(M,B)=0$. As it is located between the curves listed in the previous
paragraph, we have that $W (M,B)=0$ is contained in the set
\begin{equation}\label{eq:regionR}
\mathcal{R}=\left\{(M,B)\in\R^2:B>0,M\!>\!0,3B^2+8B\!>\!3M^2,B^2+2B\!<\!M^2,B\!>\!M-1\right\},
\end{equation}
see Figure~\ref{fig:cotasMB}.

Without loss of generality, we consider system~\eqref{sys:bt_MB}
only in the region $\mathcal{R}$. The known results about $b=b^*(m)$
can be translated to analogous results for system \eqref{sys:bt_MB}.

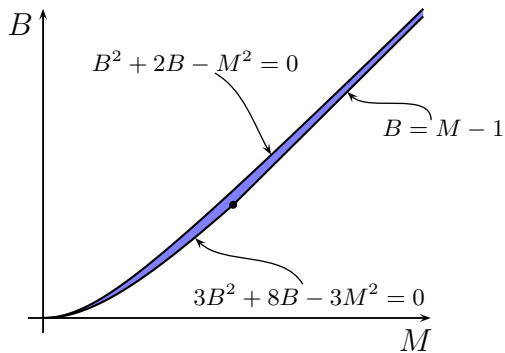
\begin{figure}[h]
\begin{center}
\psset{xunit=1.cm,yunit=1.cm,algebraic=true,arrowsize=3pt}
\begin{pspicture*}(-1,-0.5)(8.5,4.4)
\psline[linewidth=.7pt]{->}(-0.2,0)(5.1,0)
\psline[linewidth=.7pt]{->}(0,-0.2)(0,4.1)
\pscustom[fillstyle=solid,fillcolor=-green!50!red,linestyle=none]{
\psplot{5}{0}{sqrt(x^2+1)-1} \psplot{0}{2.5}{-4/3+(1/3)*sqrt(16+9*x^2)}
\psplot{2.5}{5}{x-1}}
\psplot[linestyle=solid, linewidth=.8pt,plotpoints=300]{0}{5}{sqrt(x^2+1)-1}
\psplot[linestyle=solid, linewidth=.8pt,plotpoints=300]{2.5}{5}{x-1}
\psplot[linestyle=solid, linewidth=.8pt,plotpoints=300]{0}{2.5}{-4/3+(1/3)*sqrt(16+9*x^2)}
\psdots[dotsize=3pt](2.5,1.5)
\rput(4.9,-0.3){$M$} \rput(-0.3,3.9){$B$}
\rput(2,3.4){\rnode{A}{\scriptsize{$B^2+2B-M^2=0$}}} \rput(3,2.163){\pnode{B}}
\rput(5.1,2.5){\rnode{C}{\scriptsize{$\quad B=M-1$}}} \rput(4,3){\pnode{D}}
\rput(5,4.01){\pnode{F}} \rput(3.5,0.3){\rnode{G}{\scriptsize{$3B^2+8B-3M^2=0$}}}
\rput(2,1.06){\pnode{H}} \nccurve[linewidth=0.5pt,angleA=-30,angleB=120]{->}{A}{B}
\nccurve[linewidth=0.5pt,angleA=90,angleB=-45]{->}{C}{D}
\nccurve[linewidth=0.5pt,angleA=120,angleB=45]{->}{E}{F}
\nccurve[linewidth=0.5pt,angleA=120,angleB=-45]{->}{G}{H}
\end{pspicture*}
\end{center}
\caption{\small Curves that define the region $\mathcal{R}$ (colored), that contains 
$W(M,B)=0$}\label{fig:cotasMB}
\end{figure}

\section{Bounds near the origin}\label{se:3}

This section is devoted to the study of lower and upper bounds for $b=b^*(m)$ for small values of
$m$. The first result proves the lower bound given in Perko's Conjecture~$\mathcal{II}$, by using a
suitable rational parametrization of the straight line $7b-5m=0.$ The second result provides a new
algebraic upper bound. In the $M,B$ parameters, this upper bound, $D(M,B)=0$ is given, in an
implicit way, by a polynomial $D$ of degree $14$. When we transform it into the $m,b$ variables, we
get an algebraic curve of degree $25$. In Theorem~\ref{thm:cota_global} we give a much simpler upper
bound, see the details in Section~\ref{se:5}.

Notice that the lines $7b-5m=0$ and $b=m-1$ intersect at $(m,b)=(7/2,5/2)$ and for $m>7$ we have
$m-1>5m/7$. Recall that it is known that $ m -1<b^*(m),$ see~\cite{Per92} or
Lemma~\ref{lem:rectainva}. Hence, in next result, we only consider $m\le 7/2.$

\begin{proposition}\label{pr:cubica}
For all $m\le 7/2$ it holds that $b^*(m)>5m/7$.
\end{proposition}

\begin{proof}
Recall that system~\eqref{sys:bt_MB} has a saddle point at the origin and the linear
appro\-xi\-ma\-tion of its separatrices is given by the equation
\begin{equation}\label{eq:eqdelloop_c2}
C_2(x,y):=(y-(B+M)x)(y-(B-M)x)=0,
\end{equation}
which defines two straight lines. The one with slope $B+M$ (resp. $B-M$) is tangent at the origin
to the unstable (resp. stable) separatrix.

The curve $7b-5 m =0$, for $m,b>0$, can be rationally parametrized as
\begin{equation}\label{eq:parametrization}
(m(s),b(s))=\left(\frac{7s^2}{2(6s+7)},\frac{5s^2}{2(6s+7)}\right),
\end{equation}
with $s>0.$ From the above expression we have that $M(s)=(7s+3s^2)/(6s+7)$ and $B(s)=3s^2/(6s+7)$.
Hence, with this parametrization, the slopes of the separatrices are rational functions of $s$.

For $N\in\{1,2\}$, let us consider an algebraic curve of degree $2(N+1)$ of the form
\begin{equation}\label{eq:eqdelloop_cub}
C(x,y):=C_2(x,y)+\sum\limits_{k=3}^{2(N+1)}\left(c_{k,0}\,x^k+c_{k-1,1}\,x^{k-1}y+c_{k-2,2}\,x^{k-2}
y^2\right)=0,
\end{equation}
to be determined. Notice that the above curve is quadratic in $y$.
Following the method described in \cite{GasGiaTor10}, we impose that
this expression defines a curve as close as possible to the
separatrices. It means that it should coincide at the origin
with the separatrices up to the highest possible derivative orders.
For this purpose, first we evaluate the Taylor series expansions of
both separatrices close to the origin. We express the separatrices
as functions of $x$,
\begin{equation}\label{eq:taylorsepa}
y=\Phi^{\pm}(x):=\sum\limits_{k=1}^{\infty}a_k^{\pm}x^k,
\end{equation}
where the superscript sign determines the separatrix that we approach in each
case. Then $a_k^{\pm}$ are real numbers obtained from the identity
\begin{equation*}
\left.\frac{\partial(y-\Phi^{\pm}(x))}{\partial x}y+\frac{\partial(y-\Phi^{\pm}(x))}
{\partial y}\left(\frac{7s^2}{6s+7}\,x+\frac{6s^2}{6s+7}\,y +
x^2+xy\right)\right|_{y=\Phi^{\pm}(x)}\equiv0.
\end{equation*}
Straightforward computations show that the first values of $a_k^{\pm}$ are
\begin{equation*}
\begin{aligned}
 &a_1^{+}=s, & a_1^{-}&= \,{\frac {-7s}{6\,s+7}},\\
 &a_2^{+}={\frac { \left( s+1 \right) \left( 6\,s+7 \right) }{3s \left( 4\,
s+7 \right) }}, & a_2^{-}&= {\frac {s-7}{3s \left( 2\,s+7 \right) }},\\
 &a_3^{+}=-{\frac { \left( s+1 \right) \left( 6\,s+7 \right) ^{2} \left(
5\,s+14 \right) }{18{s}^{3} \left( 4\,s+7 \right) ^{2} \left( 9\,s+14
 \right) }}, & a_3^{-}&= {\frac { \left( 7-s \right) \left( 6\,s+7 \right) ^{2} \left(
s+2 \right) }{18{s}^{3} \left( 2\,s+7 \right) ^{2} \left( 3\,s+14
 \right) }}.\\
\end{aligned}
\end{equation*}

In fact all the coefficients of $\Phi^{\pm}$ have rational expressions depending on
$s$ with non vanishing denominators when $s>0.$

Substituting \eqref{eq:taylorsepa} in \eqref{eq:eqdelloop_cub} we get
\begin{equation}\label{eq:G}
G^{\pm}(x):=C(x,\Phi^{\pm}(x))=\sum\limits_{k=3}^{\infty}g_k^{\pm}x^k.
\end{equation}
From definition \eqref{eq:eqdelloop_c2}, $g_k^{\pm}$ vanish for
$k=1,2$. Imposing that the curve defined by $C(x,y)=0$ in
\eqref{eq:eqdelloop_cub} becomes closer to the separatrices provides
extra conditions $g_k^{\pm}=0$ for higher values of $k$. We have $6N$
free coefficients $c_{i,j}$. We can fix these coefficients
imposing that $g_k^{\pm}=0,$ for $k=3,4,\ldots, 3N+2$ and solving the
system. We obtain that all the $c_{i,j}$ are rational functions of
$s$, which are well defined for $s>0.$ We call $U(x,y,s)$
the numerator of $C(x,y)$. We have
\begin{equation}\label{eq:eqdelloop_cub nova}
U(x,y,s)=T_2(x,s)y^2+T_1(x,s)y+T_0(x,s),
\end{equation}
where $T_0,T_1$ and $T_2$ are polynomials.

The proof continues showing that, for a given set of values of $s$, the algebraic curve $U(x,y,s)=0$
defines a positive invariant closed region that contains the focus point; see the right picture in
Figure~\ref{fig:lazos}. This assertion follows if we prove:
\begin{enumerate}[(I)]
\item The curve
$U(x,y,s)=0$ has a loop, as it is shown in Figure~\ref{fig:lazos}, included in the strip $\widetilde
x (s)<x<0$ for a given negative value $\widetilde x (s).$
\item This curve is without contact for the vector field on this strip.
\item The vector field points in on the loop.
\end{enumerate}

Notice that if we prove these properties for $s\in(0,7]$ we ensure that the straight line $7b-5m=0$
is a lower bound of $b=b^*(m)$ for $m\in(0,7/2]$.

Taking the curve $U(x,y,s)=0$ corresponding to $N=1,$ we can only prove the above assertion when
$s\in(0,5].$ When $s\in[5,7)$ we need to consider $N=2$. We will detail only the proof for $N=1$.
For the case $N=2$, we will describe only the differences between the two cases.

So, let us prove (I)-(III) taking $N=1$ and $s\in (0,5].$

Proof of (I). Notice that the curve $U(x,y,s)=0$ can be
written as
\[
y=\frac{-T_1(x,s)\pm\sqrt{\Delta(x,s)}}{2T_2(x,s)},
\]
where $\Delta:=T_1^2-4T_2T_0.$ Straightforward computations show that
\begin{equation}\label{eq:R4}
\Delta(x,s)=x^2\,R_4(x,s),
\end{equation}
where $R_4$ is a polynomial of degree $4$ in $x$ and of degree $30$
in $s.$ Hence, item~(I) will follow if we prove that there is a
negative value $\widetilde x (s)$ such that:
\begin{enumerate}[(i)]
\item $R_4(\widetilde x (s),s)=0$ and $R_4(x,s)>0$ in $(\widetilde x (s),0).$
\item $T_2(x,s)\ne0$ in $[\widetilde x (s),0).$
\end{enumerate}

We start proving (i). For each $s$, the coefficients $r_0(s)$ and $r_4(s)$ of minimum and
maximum degree of $R_4$ in $x$ are both positive, it has exactly
two simple negative zeros and it has no double zeros. The
non-existence of double zeros is due to the fact that
\[
\widetilde{R}_4(s)=\res(R_4(x,s),\partial R_4(x,s)/\partial x,x) \ne0,\quad s\in(0,5],
\]
where $\res(\cdot,\cdot,x)$ denotes the resultant with respect to $x;$ see for
instance~\cite{Kostr}. These properties follow studying the roots of $r_0, r_4$ and
$\widetilde{R}_4,$ that are polynomials with rational coefficients with respective degrees $30$,
$26$ and $190.$ Their Sturm sequences ensure that they have no positive roots for all
$s\in(0,5]$. We take $\widetilde x (s)$ to be the maximum of the negative zeros of $R_4(x,s).$

Let us prove (ii). The polynomial $T_2(x,s)$ has degrees $2$ and $14$ in $x$ and $s$, respectively.
Straightforward computations, using its Sturm sequence, show that the discriminant with respect to
$x$ of $T_2$ is a polynomial with rational coefficients of degree $26$ in $s$ without positive
zeros. Hence, it does not vanish when $s\in (0,5]$. It is easy to see that $T_2(0,s)\ne0$ when $s>0$
because it is a polynomial of degree $26$ with positive coefficients. For proving that
$T_2(\widetilde x (s),s)\ne0$ for $s\in(0,5]$ first we show that $\res(R_4(x,s),T_2(x,s),x)\ne0$ for
every $s$ in this interval. In fact, it is a polynomial of degree $106$ in $s$ with no real roots in
$s\in(0,5]$. Therefore, for these values of $s$, the number of real roots of $T_2(x,s)$ in
$[\widetilde x (s),0)$ does not depend of~$s$. Studying for instance the case $s=1$ we can easily
verify that $T_2(x,1)$ has no real roots in $[\widetilde x (1),0)$ and then we have the desired
result.

Proof of (II). We have to show that the vector field
\eqref{sys:bt_MB} is never tangent to $U(x,y,s)=0.$ To
prove this fact we study the common zeros between $U$
and its derivative with respect to the vector field
\begin{equation*}
\dot {U}(x,y,s):=\frac{\partial U(x,y,s)}{\partial x}y+\frac{\partial U(x,y,s)} {\partial
y}\left(\frac{7s^2}{6s+7}\,x+\frac{6s^2}{6s+7}\,y + x^2+xy\right).
\end{equation*}
We get
\begin{equation}\label{eq:S4}
\res(U(x,y,s),\dot{U}(x,y,s),y)=x^{12}S_{4}(x,s),
\end{equation}
where $S_{4}(x,s)$ is a polynomial of degree $4$ in $x$ and of
degree $65$ in $s$. We want to prove that $S_4$ does not change sign
for $x\in[\widetilde x (s),0)$ and $s\in(0,5]$.

For a given value of $s,$ say $s=1$, the result follows directly from the Sturm method. After that,
we proceed like in the proof of (I). We compute
\[
\res(R_4(x,s),S_4(x,s),x)\quad\mbox{and}\quad
 \res(S_4(x,s),\partial S_4(x,s)/\partial x,x).
\]
We obtain two polynomials in $s$ of degrees $362$ and $438$, respectively. In the interval $(0,5],$
the former polynomial does not vanish and the latter has six different real roots. Then, as can be
seen with the Sturm method, in each of the seven intervals defined by these roots the relative
position of the maximum negative zero of $R_4(x,s)$ with respect to the negatives zeros of
$S_4(x,s)$ does not change. Choosing $s$ in each interval we can check that there are no zeros
of $S_4$ in $[\widetilde x (s),0)$ when $s\in(0,5],$ as we wanted to prove. Moreover, $S_4(0,s)$
does not vanish on $s\in(0,5].$ Then the curve $U(x,y,s)=0$ is without contact in the region where
the loop is defined. 

We remark that this step is the one that does not work in the whole interval $(0,7)$, taking $N=1$,
because $S_4(0,s)$ has a zero close to $5.08.$

Proof of (III). We finish the case $N=1$ checking that the vector field points in at the
intersection point $(x_0(s),0),$ of the loop contained in the curve $U(x,y,s)=0$ with the $x$-axis.
We prove that $\dot{U}(x_0(s),0,s)$, $\frac{\partial U}{\partial x}(x_0(s),0,s)$ and $\frac{\partial
U}{\partial y}(x_0(s),0,s)$ are all negative for $s\in (0,5]$, where $x_0(s)$ is the biggest
negative zero of $U(x,0,s).$ As in the above items, it is enough to check the inequalities for a
concrete value of $s$. This holds because, straightforward computations show that $\res
(U(x,0,s)/x^2,\dot {U}(x,0,s)/x^2,x),$ which is a polynomial of degree $69$ in $s$, does not vanish
in $(0,5],$ as can be verified with the Sturm method.

As we have already said, the proof finishes taking $N=2$ in \eqref{eq:eqdelloop_cub} and following
the above procedure for $s\in[5,7)$. All the qualitative properties remain unchanged. In fact,
expressions \eqref{eq:R4} and \eqref{eq:S4} change to $\Delta(x,s)\!=\!x^2R_8(x,s)$
and $\res(U(x,y,s),\dot{U}(x,y,s),y)=x^{18}S_{8}(x,s),$ respectively. Here $R_8$ and $S_8$ are now
polynomials of degree $8$ in $x$ and $98$ and $229$ in $s$, respectively. 
\end{proof}

\begin{proposition}\label{prop:M_pequeno}
For all $\,\,0<M\le30,$ the graph of $W (M,B)=0$ is below the graph
of the function defined by the branch of the algebraic curve
$D(M,B)=0$, that writes as
\begin{align*}
&\!-\!{}383292 M^{14}\!-\!{}1910439 M^{13}B\!-\!{}3223665 M^{12}B^{2}\!-\!{}314748
M^{11}B^{3}\!+\!{}5603940 M^{10}B^{4}\\
&\!+\!{}5541141 M^{9}B^{5}\!-\!{}2323401 M^{8}B^{6}\!-\!{}7154664 M^{7}B^{7}\!-\!{}3397092
M^{6}B^{8}\!+\!{}2020587 M^{5}B^{9}\\
&\!+\!{}3218997 M^{4}B^{10}\!+\!{}1742052 M^{3}B^{11}\!+\!{}499644 M^{2}B^{12}\!+\!{}76071
MB^{13}\!+\!{}4869 B^{14}\\
&\!-\!{}500742 M^{13}\!-\!{}787023 M^{12}B\!+\!{}4493070 M^{11}B^{2}\!+\!{}14795091
M^{10}B^{3}\!+\!{}11566572 M^{9}B^{4}\\
&\!-\!{}11585754 M^{8}B^{5}\!-\!{}24443044 M^{7}B^{6}\!-\!{}8307134 M^{6}B^{7}\!+\!{}12772706
M^{5}B^{8}\\
&\!+\!{}16289545 M^{4}B^{9}\!+\!{}8662166 M^{3}B^{10}\!+\!{}2509195 M^{2}B^{11}\!+\!{}388536
MB^{12}\!+\!{}25344 B^{13}\\
&\!-\!{}174798 M^{12}\!+\!{}1420524 M^{11}B\!+\!{}7005177 M^{10}B^{2}\!+\!{}4483350
M^{9}B^{3}\!-\!{}16943919 M^{8}B^{4}\\
&\!-\!{}28501282 M^{7}B^{5}\!-\!{}3691132 M^{6}B^{6}\!+\!{}27741570 M^{5}B^{7}\!+\!{}31479694
M^{4}B^{8}\!+\!{}16742014 M^{3}B^{9}\\
&\!+\!{}4938651 M^{2}B^{10}\!+\!{}781536 MB^{11}\!+\!{}52119 B^{12}\!-\!{}48600 M^{11}\!+\!{}855846
M^{10}B\!+\!{}404136 M^{9}B^{2}\\
&\!-\!{}8473533 M^{8}B^{3}\!-\!{}14178838 M^{7}B^{4}\!+\!{}2903273 M^{6}B^{5}\!+\!{}26313718
M^{5}B^{6}\!+\!{}28894211 M^{4}B^{7}\\
&\!+\!{}15714446 M^{3}B^{8}\!+\!{}4769205 M^{2}B^{9}\!+\!{}775554 MB^{10}\!+\!{}53046
B^{11}\!+\!{}226800 M^{9}B\\
&\!-\!{}1138914 M^{8}B^{2}\!-\!{}2990748 M^{7}B^{3}\!+\!{}2116351 M^{6}B^{4}\!+\!{}10975549
M^{5}B^{5}\\
&\!+\!{}12542602 M^{4}B^{6}\!+\!{}7156446 M^{3}B^{7}\!+\!{}2261723 M^{2}B^{8}\!+\!{}380289
MB^{9}\!+\!{}26766 B^{10}\\
&\!-\!{}264600 M^{7}B^{2}\!+\!{}218442 M^{6}B^{3}\!+\!{}1575182 M^{5}B^{4}\!+\!{}2042992
M^{4}B^{5}\!+\!{}1262428 M^{3}B^{6}\\
&\!+\!{}421554 M^{2}B^{7}\!+\!{}73806 MB^{8}\!+\!{}5364 B^{9}=0,
\end{align*}
and whose series expansion at $M=0$ is
\begin{equation*}
B=\frac{3}{7} M^{2}-\frac {180}{2401} M^{4}+\frac{2366307}{90589730} M^{6}+O(M^7).
\end{equation*}
\end{proposition}

\begin{proof}
It follows using the same ideas and techniques developed in the proof of
Proposition~\ref{pr:cubica} but with bigger computational difficulties. We only
comment the main differences avoiding the details.

Here it is sufficient to consider a curve $C(x,y)=0$ like in
\eqref{eq:eqdelloop_cub} of degree four. Nevertheless, its
coefficients are obtained in a different way. We solve the system
given by $g_k^+=0, k=3,4,5,6$ and $g_3^-=g_4^-=0.$ It means that we
match the separatrices with different orders. Finally, using the same
notation that in the proof of Proposition~\ref{pr:cubica}, the
expression of $D(M,B)=0$ is obtained imposing the condition
$g^+_7(M,B)=0,$ following the same strategy as in~\cite{GasGiaTor10}.
\end{proof}

\begin{remark}\label{new}
In the $m,b$ variables the expression of $D(M,B)=0$ provided in
the latter result is transformed into a new algebraic curve of
degree $25$ with $257$ monomials.
\end{remark}

\section{Bounds up to infinity}\label{se:4}

For the sake of completeness we also include a proof of the inequality $b^*(m)>m-1,$ different to
the one given in \cite{Per92}. It is based on the
Bendisxon-Dulac criterion, and it is quite simple.

\begin{lemma}\label{lem:rectainva}
For all $m>0$, it holds that $b^*(m)>m-1.$
\end{lemma}

\begin{proof}
It suffices to prove that, when $b= m -1$, system \eqref{sys:bt_original} has no limit cycles. It
writes as
\begin{equation*}
\left\{{\begin{array}{l}
x'=y:=P(x,y),\\
y'=-m^2+(m-1)y+x^2+xy:=Q(x,y),
\end{array}}\right.
\end{equation*}
and it has the invariant straight line $\ell:=\{x+y- m =0\}.$
Therefore the limit cycles, if they exist, are contained in
$\mathcal{U}:=\mathbb{R}^2\setminus\ell.$ We have that on
$\mathcal{U},$
\begin{equation*}
\diver\left(\frac {P(x,y)}{L(x,y)},\frac{Q(x,y)}{L(x,y)}\right)
=\frac{\partial}{\partial x}\left(\frac{P(x,y)}{L(x,y)}\right)+
\frac{\partial}{\partial y}\left(\frac{Q(x,y)}{L(x,y)}\right)=
\frac{-1}{L(x,y)}\ne0,
\end{equation*}
where $L(x,y)=x+y-m.$ Therefore, we can apply the well-known Bendixon-Dulac criterion
(\cite{Per}) to each one of the half planes of $\mathcal{U}$,
obtaining that the system has no limit cycles, as we wanted to
prove.
\end{proof}

As a first step to obtain the upper bound of $b^*(m)$ for $m\ge7$
given in Theorem~\ref{thm:cota_global} we study the curve $W (M,B)=0$ on regions $\{(M,B)\,:\,
M>M_{\alpha}\}$ that arrive to infinity. The procedure is similar to the one
described in Section \ref{se:3}, but in this case only upper bounds will be provided
because we only have been able to obtain negatively invariant regions. As we will see, these bounds
will be enough to prove Perko's Conjecture~$\mathcal{I}$.

Suitable negatively invariant regions are much more difficult to
be found that in the previous section. We will construct them using piecewise algebraic curves;
see Figure~\ref{fig:Ftrozos}.

\begin{figure}[h]
\begin{center}
\psset{xunit=0.8cm,yunit=0.8cm,algebraic=true,arrowsize=3pt}
\begin{pspicture*}(-3.5,-3)(0.5,2)
\psline[linewidth=.7pt]{->}(-3.5,0)(0.5,0)
\psline[linewidth=.7pt]{->}(0,-3)(0,2)
\psdots[dotsize=3pt](0,0)(-2,0)
\psdots[dotsize=5pt,dotstyle=+,dotangle=45](-2.99,0)
\psdots[dotsize=5pt,dotstyle=+](-2,1.56)
\psplot[linestyle=solid,linewidth=1.3pt,plotpoints=1500]{-3}{0}{
2/9*1/(8256+141*sqrt(226)+266*x)*((11008*x+188*x*sqrt(226)+33024+564*sqrt(226)+sqrt(158811104*x^2 +
4645339*x^2*sqrt(226)+1844762403*x+57783633*x*sqrt(226)+4104987273+131542848*sqrt(226)))*x)}
\psplot[linestyle=solid,linewidth=1.3pt,plotpoints=1500]{-3}{-2}{
4/3-(1/6)*sqrt(226)+(4/9)*x-(1/18)*x*sqrt(226)+(1/18)*sqrt(-39150+2160*sqrt(226) - 43500*x+2400*x
* sqrt(226)-10150*x^2+560*x^2*sqrt(226))}
\psplot[linestyle=solid,linewidth=1.3pt,plotpoints=1500]{-2}{0}{(8/9)*x-(1/9)*x*sqrt(226)}
\end{pspicture*}
\end{center}
\caption{\small Piecewise loop for obtaining a negatively invariant
region}\label{fig:Ftrozos}
\end{figure}
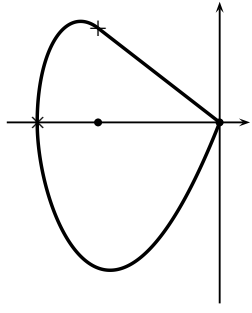

\begin{proposition}\label{prop:M_grande}
For every real $\alpha>0$, there exists $M_{\alpha}>\sqrt{\alpha}>0$ such that
\[
\{W
(M,B)=0\}\subset\left\{M-1<B<M-1+\frac{\alpha}{M^2}\right\}\cap\{M>M_\alpha\}.
\]
\end{proposition}

\begin{proof}
We already know that when $B=M-1$ the system has no limit cycles, so we only need
to prove that, for every $\alpha$, there exist a curve $B=M-1+\alpha/M^2$ and a
value $M_\alpha$ such that on this curve the system has a limit cycle when
$M>M_\alpha$. So we will assume that $B=M-1+\alpha/M^2$.

We propose a curve $C(x,y)=0$ formed by three pieces $F_i$ of
different algebraic curves of degree $i$, for $i=1,2,3.$
Figure~\ref{fig:Fi} shows the shape of the curve $C(x,y)=0$ and the
corresponding pieces. For each $i$, we use the same $F_i$ to denote
the algebraic curve $F_i(x,y)=0$ that contains the corresponding
piece.

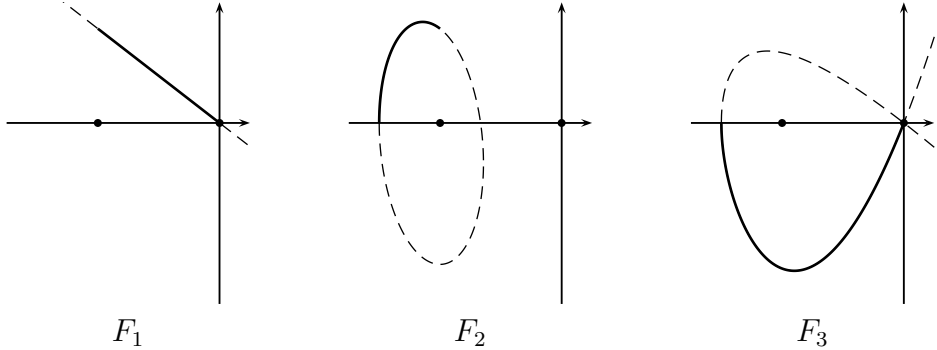
\begin{figure}[h]
\begin{center}
\begin{tabular}{c@{\hspace{1.3cm}}c@{\hspace{1.3cm}}c}
\psset{xunit=0.8cm,yunit=0.8cm,algebraic=true,arrowsize=3pt}
\begin{pspicture*}(-3.5,-3)(0.5,2)
\psline[linewidth=.7pt]{->}(-3.5,0)(0.5,0)
\psline[linewidth=.7pt]{->}(0,-3)(0,2)
\psdots[dotsize=3pt](0,0)(-2,0)
\psplot[linestyle=dashed,linewidth=0.5pt,plotpoints=1500]{-3}{1}{(8/9)*x-(1/9)*x*sqrt(226)}
\psplot[linestyle=solid,linewidth=1pt,plotpoints=1500]{-2}{0}{(8/9)*x-(1/9)*x*sqrt(226)}
\end{pspicture*}
& \psset{xunit=0.8cm,yunit=0.8cm,algebraic=true,arrowsize=3pt}
\begin{pspicture*}(-3.5,-3)(0.5,2)
\psline[linewidth=.7pt]{->}(-3.5,0)(0.5,0)
\psline[linewidth=.7pt]{->}(0,-3)(0,2)
\psdots[dotsize=3pt](0,0)(-2,0)
\psplot[linestyle=dashed,linewidth=0.5pt,plotpoints=1500]{-3}{-1.286}{
4/3-(1/6)*sqrt(226)+(4/9)*x-(1/18)*x*sqrt(226)+(1/18)*sqrt(-39150+2160*sqrt(226)-43500*x
+ 2400*x*sqrt(226)-10150*x^2+560*x^2*sqrt(226))}
\psplot[linestyle=dashed,linewidth=0.5pt,plotpoints=1500]{-3}{-1.286}{
4/3-(1/6)*sqrt(226)+(4/9)*x-(1/18)*x*sqrt(226)-(1/18)*sqrt(-39150+2160*sqrt(226)-43500*x
+ 2400*x*sqrt(226)-10150*x^2+560*x^2*sqrt(226))}
\psplot[linestyle=solid,linewidth=1pt,plotpoints=1500]{-3}{-2}{
4/3-(1/6)*sqrt(226)+(4/9)*x-(1/18)*x*sqrt(226)+(1/18)*sqrt(-39150+2160*sqrt(226)-43500*x
+ 2400*x*sqrt(226)-10150*x^2+560*x^2*sqrt(226))}
\end{pspicture*}
& \psset{xunit=0.8cm,yunit=0.8cm,algebraic=true,arrowsize=3pt}
\begin{pspicture*}(-3.5,-3)(0.5,2)
\psline[linewidth=.7pt]{->}(-3.5,0)(0.5,0)
\psline[linewidth=.7pt]{->}(0,-3)(0,2)
\psdots[dotsize=3pt](0,0)(-2,0)
\psplot[linestyle=solid,linewidth=1pt,plotpoints=1500]{-3}{0}{
2/9*1/(8256+141*sqrt(226)+266*x)*((11008*x+188*x*sqrt(226)+33024+564*sqrt(226)+sqrt(158811104*x^2
+ 4645339*x^2*sqrt(226)+1844762403*x+57783633*x*sqrt(226)+4104987273+131542848*sqrt(226)))*x)}
\psplot[linestyle=dashed,linewidth=0.5pt,plotpoints=1500]{-3}{0.5}{
2/9*1/(8256+141*sqrt(226)+266*x)*((11008*x+188*x*sqrt(226)+33024+564*sqrt(226)+sqrt(158811104*x^2
+ 4645339*x^2*sqrt(226)+1844762403*x+57783633*x*sqrt(226)+4104987273+131542848*sqrt(226)))*x)}
\psplot[linestyle=dashed,linewidth=0.5pt,plotpoints=1500]{-3}{0.5}{
2/9*1/(8256+141*sqrt(226)+266*x)*((11008*x+188*x*sqrt(226)+33024+564*sqrt(226)-sqrt(158811104*x^2
+ 4645339*x^2*sqrt(226)+1844762403*x+57783633*x*sqrt(226)+4104987273+131542848*sqrt(226)))*x)}
\end{pspicture*}\\
$F_1$&$F_2$&$F_3$
\end{tabular}
\end{center}
\caption{\small The different pieces $F_i,i=1,2,3$}\label{fig:Fi}
\end{figure}

The first one is
\[
F_1:=\{(x,(B-M)x)\,:\, B^2-M^2\le x\le0\}.
\]
Notice that it is given by the segment of the straight line tangent
to the stable separatrix that starts at the origin and ends at the
point $(x_1,y_1)$, where $x_1=B^2-M^2$ is the $x$-coordinate of the
focus. Hence $F_1(x,y)= y-(B-M)x.$

We take $F_2$ as a portion of the quadratic curve
$F_2(x,y)=1+a_{1,0}x+a_{0,1}y+a_{2,0}x^2+a_{1,1}xy+a_{0,2}y^2=0$
that passes trough $(x_1,y_1)$, is tangent to $F_1$ at this point,
passes also by the point $(x_2,0)$, with $x_2=-3M$, and coincides
at this point, until second order derivatives, with the solution of
the differential equation. More concretely, $F_2$ is the piece
between $(x_1,y_1)$ and $(x_2,0).$ The choice of this value for
$x_2$ is motivated in Remark~\ref{re:ab}.

The third piece $F_3$ is contained in a cubic curve similar to the
one defined in \eqref{eq:eqdelloop_cub}.
 We consider
\[
F_3(x,y)=C_2(x,y)+c_{3,0}x^3+c_{2,1}x^2y+c_{1,2}xy^2+c_{0,3}y^3=0.
\]
This curve is tangent to both separatrices at the origin. The four
coefficients of the homogeneous part of degree three will be fixed
to get that $F_3$ approaches two times more the unstable separatrix
at the origin, passes trough $(x_2,0)$ and be tangent to the
solution of the differential equation passing by this point. In
fact, $F_3$ is the piece of the curve between the origin and
$(x_2,0)$ which is contained in the third quadrant.

Let us study now the behavior of the vector field on the three pieces. The gradient of $F_1$ at
$(x_1,y_1)$ is $(M-B,1)$, the gradients of $F_2$ and $F_3$ at $(x_2,0)$ are
$\left(\frac{2(-B^2-3M+M^2)^3}{3(B+M)^2(M-B)^4M},0\right)$ and $\left(-3(B+M)(M-B)M,0\right),$
respectively. If $M>\sqrt{\alpha}$ we can conclude that the gradient of $C$ points to the exterior
of the closed curve $C(x,y)=0$. Thus, if we prove that the algebraic curves defined by $F_i(x,y)=0$
and $\dot{F}_i(x,y)=0,$ the derivative of $F_i$ with respect to the vector field, have no common
points, we can easily conclude that the vector field points to the exterior of the curve $C(x,y)=0$
for $M>\sqrt\alpha.$

The result for the segment of straight line $F_1$ is straightforward.

To study the vector field on $F_2$ we compute $R_2:=\res(F_2,\dot{F}_2,y)$. We get that
$R_2=(x-x_2)^2p_4(x,M,\alpha)$, where $p_4$ is a polynomial of degree
four in $x$. At $M=\infty$, this polynomial has an asymptotic
expansion with dominant term $-M^{-6}(2M+x)^3(4M+x)/46656.$ Hence,
the dominant terms of the asymptotic expansions of the
corresponding roots are $-4M$ and $-2M$. This latter value
corresponds to a triple root. A more detailed computation,
considering the next significant term in each coefficient of $p_4$,
shows that the triple root splits into a couple of complex
conjugated roots and a real one, $\widehat x_1$, with asymptotic
expansion
\[\widehat x_1(M,\alpha)=
-2M+\frac{2\sqrt[3]{\alpha}}{\sqrt[3]{9}}M^{2/3}+o(M^{2/3}).\] Since
the other real root of $p_4$ has asymptotic expansion $\widehat
x_2(M,\alpha)=-4M+o(M),$
\[x_1=B^2-M^2=\Big(M-1+\frac{\alpha}{M^2}\Big)^2-M^2=-2M+1+O\Big(\frac1M\Big),\]
and $x_2=-3M$, it holds that for $M>\widetilde{M}_2(\alpha)$ big enough,
\[
\widehat x_2(M,\alpha)<x_2<x_1<\widehat x_1(M,\alpha).
\]
Therefore, for $M>\widetilde{M}_2(\alpha)$, $R_2$ does not change
sign on $(x_2,x_1)$, as we wanted to prove. Moreover, it is easy to
see that the vector field points to the exterior of the loop on
$F_2$.

The resultant $R_3$ of $F_3$ and $\dot{F}_3$ with respect to $x$ is
of the form $y^9q_3(y,M,\alpha)$, where $q_3$ is a polynomial of degree
three in $y$. The asymptotic expansion of the ordered coefficients of
$q_3$ at $M=\infty$ shows the sign configuration $[+,-,+,-]$ for
$M>\widetilde{M}_3(\alpha)$ large enough. Thus, it is clear that $q_3$
has no negative zeros. Then, for $M>\widetilde{M}_3(\alpha),$ $R_3$
does not vanish in the half-plane $y<0$ where $F_3$ is defined.
Similarly to the previous cases, we can check that the vector field
points to the exterior of $C(x,y)=0$ along $F_3$.

\begin{figure}[h]
\begin{center}
\psset{xunit=0.8cm,yunit=0.8cm,algebraic=true,arrowsize=3pt}
\begin{pspicture*}(-3.5,-3)(0.5,2)
\psline[linewidth=.7pt]{->}(-3.5,0)(0.5,0)
\psline[linewidth=.7pt]{->}(0,-3)(0,2)
\psdots[dotsize=3pt](0,0)(-2,0)
\psplot[linestyle=solid,linewidth=1pt,plotpoints=1500]{-3}{0}{
2/9*1/(8256+141*sqrt(226)+266*x)*((11008*x+188*x*sqrt(226)+33024+564*sqrt(226)+sqrt(158811104*x^2
+ 4645339*x^2*sqrt(226)+1844762403*x+57783633*x*sqrt(226)+4104987273+131542848*sqrt(226)))*x)}
\psplot[linestyle=solid,linewidth=1pt,plotpoints=1500]{-3}{-2}{
4/3-(1/6)*sqrt(226)+(4/9)*x-(1/18)*x*sqrt(226)+(1/18)*sqrt(-39150+2160*sqrt(226)-43500*x+2400*x
*sqrt(226)-10150*x^2+560*x^2*sqrt(226))}
\psplot[linestyle=solid,linewidth=1pt,plotpoints=1500]{-2}{0}{(8/9)*x-(1/9)*x*sqrt(226)}
\psline[linestyle=solid,linewidth=0.5pt]{->}(-3.02,-0.2)(-3.02,0.7)
\psline[linestyle=solid,linewidth=0.5pt]{->}(-2.85,0.9)(-2.65,1.7)
\psline[linestyle=solid,linewidth=0.5pt]{->}(-2.3,1.5)(-1.7,1.6)
\psline[linestyle=solid,linewidth=0.5pt]{->}(-1.3,0.85)(-0.7,0.75)
\psline[linestyle=solid,linewidth=0.5pt]{->}(-0.4,0.22)(0,0.1)
\psline[linestyle=solid,linewidth=0.5pt]{->}(-2.6,-1.5)(-3,-0.9)
\psline[linestyle=solid,linewidth=0.5pt]{->}(-1.7,-2.35)(-2.3,-2.45)
\psline[linestyle=solid,linewidth=0.5pt]{->}(-0.9,-1.5)(-1.1,-2.3)
\psline[linestyle=solid,linewidth=0.5pt]{->}(-0.1,-0.1)(-0.3,-0.9)
\end{pspicture*}
\end{center}
\vspace{-3mm} \caption{\small The negatively invariant region corresponding to the
piecewise loop}\label{fig:campoFtrozos}
\end{figure}
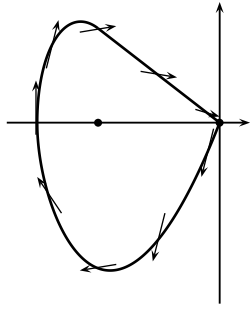

The proof finishes taking $M_{\alpha}$ as the maximum of $\widetilde{M}_2(\alpha)$,
$\widetilde{M}_3(\alpha)$ and $\sqrt{\alpha}.$ Then we have the situation given
in Figure~\ref{fig:campoFtrozos}.
\end{proof}

Next remark clarifies some details about the previous proof. Lemma~\ref{le:disc} will follow
directly applying the above result for a concrete value of $\alpha.$ The corresponding $M_{\alpha}$
is also given.

\begin{remark}\label{re:ab}
The choice of the point $(x_2,0)=(-3M,0)$ in the proof of
Proposition~\ref{prop:M_grande} is motivated by some numerical
computations. We wanted to choose a point on the $x$-axis that was
between the first crossing point of the separatrices of the saddle
point with the negative $x$-axis. Let us call $(P_s,0)$ the first
crossing point for the stable separatrix, and $(P_u,0)$ for the
unstable one. For several values of $\alpha$, using a
Runge-Kutta-Fehlberg 4-5 method, together with degree four
interpolation, we obtain numerical approximations of the points
$P_s$ and $P_u$ for every $M$. Figure~\ref{fig:separatrices}(b) shows
the plots of $-P_s/M$ and $-P_u/M$ together with the corresponding
value of the abscissa of the focus point. We remark that for the values of
$\alpha$ that we have checked the limit behavior is always the
same. Hence, the asymptotic expansions at $M=\infty$ of $(P_s,0)$ and
$(P_u,0)$ seem to be $(-2M,0)$ and $(-4M,0)$, respectively; see
Figure~\ref{fig:separatrices}. So, our choice of $x_2=-3M$ is quite
natural.
\end{remark}

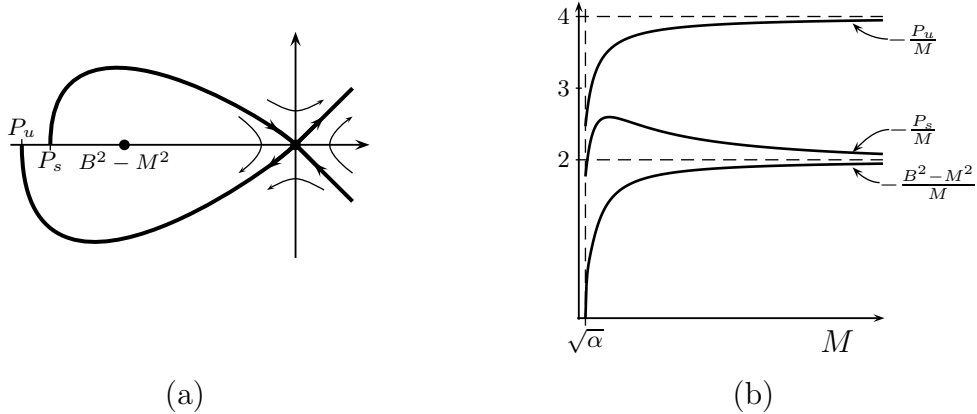
\begin{figure}[h]
\begin{center}
\begin{tabular}{c@{\hspace{2cm}}c}
\raisebox{.5\height}{
\psset{xunit=1.5cm,yunit=1.5cm,algebraic=true,arrowsize=4pt}
\begin{pspicture*}(-2.6,-1.1)(0.75,1.1)
\psline[linewidth=.7pt]{->}(-2.5,0)(0.65,0)
\psline[linewidth=.7pt]{->}(0,-1)(0,1)
\psdots[dotsize=4pt](0,0)(-1.5,0)
\psbezier[linewidth=1.5pt,plotpoints=500]{-}(-2.4,0)(-2.4,-1.8)(-0.25,-0.25)(0,0)
\psbezier[linewidth=1.5pt,plotpoints=500]{-}(0,0)(0.25,0.25)(0.2,0.2)(0.5,0.5)
\psbezier[linewidth=1.5pt,plotpoints=500]{<->}(-0.25,-0.21)(0,0)(0,0)(0.25,0.25)
\psline[linewidth=.5pt]{-}(-2.4,0.05)(-2.4,-0.05)
\psbezier[linewidth=1.5pt,plotpoints=500]{-}(-2.15,0)(-2.15,1.4)(-0.25,0.25)(0,0)
\psbezier[linewidth=1.5pt,plotpoints=500]{-}(0,0)(0.25,-0.25)(0.2,-0.2)(0.5,-0.5)
\psbezier[linewidth=1.5pt,plotpoints=500]{>-<}(-0.25,0.19)(0,0.025)(0,0)(0.25,-0.25)
\psline[linewidth=.5pt]{-}(-2.15,0.05)(-2.15,-0.05)
\pscurve[linewidth=0.5pt,arrowsize=2pt,plotpoints=500]{<-}(-0.25,-0.4)(0,-0.3)
(0.25,-0.4)
\pscurve[linewidth=0.5pt,arrowsize=2pt,plotpoints=500]{->}(-0.25,0.4)(0,0.3)
(0.25,0.4)
\pscurve[linewidth=0.5pt,arrowsize=2pt,plotpoints=500]{->}(-0.5,0.25)(-0.3,0)
(-0.5,-0.25)
\pscurve[linewidth=0.5pt,arrowsize=2pt,plotpoints=500]{->}(0.5,-0.25)(0.3,0)
(0.5,0.25)
\rput(-1.5,-0.15){\tiny{$B^2-M^2$}} \rput(-2.15,-0.15){\scriptsize{$P_s$}}
\rput(-2.4,0.15){\scriptsize{$P_u$}}
\end{pspicture*}}
&
\psset{xunit=0.2cm,yunit=1cm,algebraic=true,arrowsize=3pt}
\begin{pspicture*}(-3,-1)(26,4.3)
\psline[linewidth=.7pt]{->}(-0.1,0)(20.1,0)
\psline[linewidth=.7pt]{->}(0,-0.02)(0,4.2)
\psline[linewidth=.5pt,linestyle=dashed]{-}(-0.2,4)(20,4)
\psline[linewidth=.5pt]{-}(-0.2,3.05)(0.2,3.05)
\psline[linewidth=.5pt,linestyle=dashed]{-}(-0.2,2.1)(20,2.1)
\psline[linewidth=.5pt,linestyle=dashed]{-}(0.437213,4.1)(0.437213,-0.1)
\psplot[linestyle=solid, linewidth=1pt,
plotpoints=300]{0.437213}{20}{16*x/(4*x+1)} \psplot[linestyle=solid,
linewidth=1pt, plotpoints=300]{0.437213}{20}{2*x*(x+4)/(x+1)^2}
\psplot[linestyle=solid, linewidth=1pt,
plotpoints=300]{0.437213}{20}{(1/25)*(5*x^2-1)*(10*x^3-5*x^2+1)/x^5+0.1}
\rput(17,-0.3){$M$} \rput(0.4,-0.3){\scriptsize{$\sqrt{\alpha}$}}
\rput(-1,2.1){\scriptsize{$2$}} \rput(-1,3.05){\scriptsize{$3$}}
\rput(-1,4){\scriptsize{$4$}}
\rput(22,3.7){\rnode{A}{\scriptsize{$-\frac{P_u}{M}$}}} \rput(18,3.91){\pnode{B}}
\rput(22,2.5){\rnode{C}{\scriptsize{$-\frac{P_s}{M}$}}} \rput(18,2.2){\pnode{D}}
\rput(23,1.8){\rnode{E}{\scriptsize{$-\frac{B^2-M^2}{M}$}}} \rput(18,2){\pnode{F}}
\nccurve[linewidth=0.5pt,angleA=180,angleB=-45]{->}{A}{B}
\nccurve[linewidth=0.5pt,angleA=180,angleB=45]{->}{C}{D}
\nccurve[linewidth=0.5pt,angleA=180,angleB=-45]{->}{E}{F}
\end{pspicture*}\\[-10pt]
(a)&(b)
\end{tabular}
\end{center}
\caption{\small Separatrices of the saddle point}\label{fig:separatrices}
\end{figure}

\section{Proof of Theorems~\ref{thm:n_grande} and \ref{thm:cota_global}}\label{se:5}
To prove Theorem~\ref{thm:n_grande} we need to translate the results of
Proposition~\ref{prop:M_grande} to the $m,b$ parameters.

\begin{lemma}\label{le:BM_bm}
The change of variables \eqref{cv:nb_a_MB} converts the curve $B=M-1+\alpha /M^{2}$
into the branch of the curve
\begin{equation}\label{eq:BMtobm}
\begin{aligned}
-&4m^5-12m^4b-8m^3b^2+8m^2b^3+12mb^4+4b^5-60m^4-48m^3b+88m^2b^2\\
+&80mb^3+4b^4+(-192-16\alpha)m^3+(384-48\alpha)m^2b+(64-48\alpha)mb^2\\
-&16\alpha b^3+(256-160\alpha)m^2-192\alpha mb-32\alpha b^2-256\alpha m+64\alpha^2=0,
\end{aligned}
\end{equation} that satisfies
\begin{equation}\label{eq:BMtobm2}
b=m-1+\frac{2\alpha}{m}- \frac{\alpha}{m^{2}}+O\left(\frac{1}{m^{3}}\right),
\end{equation} when $m$ goes to $\infty$.
Moreover, it remains below the curve $b=m-1+2\alpha/ m$ for every positive $m$ and
$\alpha$.
\end{lemma}

\begin{proof}
Straightforward computations give~\eqref{eq:BMtobm}. To prove the
last assertion it suffices to see that both curves do not have common
points. Clearly, from~\eqref{eq:BMtobm2}, the result holds for $m$
big enough.

The common points of both curves are characterized by the roots of the polynomial obtained
substituting $b=m-1+2\alpha /m$ in \eqref{eq:BMtobm}. Removing the denominator we obtain
\begin{align*}
q(m):=&\,8{m}^{7}+ \left( 16\alpha+12 \right) {m}^{6}+ \left( 24\alpha+6 \right) {m}^{5}+ 
\left(48{\alpha}^{2}-12\alpha+1 \right) {m}^{4}\\
&\,-8\alpha \left( 3\alpha+1 \right) {m}^{3}+24{\alpha}^{2} \left( 2\alpha+1 \right)
{m}^{2}-32m{\alpha}^{3}+16{\alpha}^{4}.
\end{align*}
The computation of its Sturm sequence evaluated at $0$ and at $+\infty$ gives the
configurations of signs $[+,-,-,-,+,+,-,-]$ and $[+,+,+,-,-,+,-,-]$, respectively,
for each positive $\alpha$. Hence, $q(m)$ does not have positive roots, and
therefore both curves have no common points when $m>0$ and $\alpha>0$.
\end{proof}

\begin{proof}[Proof of Theorem~\ref{thm:n_grande}]
From Proposition~\ref{prop:M_grande} and Lemma~\ref{le:BM_bm}, for each $\alpha>0,$
the curve $b= m -1+2\alpha/ m $ is an upper bound of $b=b^*(m),$ for $m>m_\alpha$,
where $m_\alpha$ can be obtained applying the transformation~\eqref{cv:nb_a_MB} to
the region $M>M_\alpha.$ Therefore, for $m>m_{\alpha}$,
\[
 m -1<b^*(m)< m -1+\frac{2\alpha}{ m },
\]
which implies
\[
\lim_{m\to\infty}\frac{b^*(m)- m +1}{1/ m }=0,
\]
because $\alpha$ is an arbitrary positive number.
\end{proof}

Before proving Theorem \ref{thm:cota_global} we
need a preliminary result.

\begin{lemma}\label{le:disc}
It holds that
\[
b^*(m)< m -1+\frac{51}{20 m },
\]
for $m>\widetilde {m},$ where $\widetilde {m}\approx 6.93$ is the unique
positive root of the polynomial
\[
\begin{aligned}
 &50331648000000000 \, {m}^{17}-243269632000000000 \, {m}^{16}\\
-&2129238425600000000 \, {m}^{15} -7211878973440000000 \, {m}^{14} \\
+&111668173209600000000 \, {m}^{13} +264470739812352000000 \, {m}^{12} \\
-&130466347912396800000 \, {m}^{11} -9197101546824499200000 \, {m}^{10}\\
-&6302900112535388160000 \, {m}^{9} +6325778059290335232000 \, {m}^{8}\\
+&2289016716587559936000 \, {m}^{7} -46572462911915224012800 \, {m}^{6} \\
+&8515659923453703340800 \, {m}^{5} -4901243812728523876800 \, {m}^{4}\\
-&45716337137659722706080 \, {m}^{3} +6052551315638078774880 \, {m}^{2}\\
-&8203038242422388605200 \, m-7953608649353382254007.
\end{aligned}
\]
\end{lemma}

\begin{proof}
Proposition~\ref{prop:M_grande}, with $\alpha=51/40$, ensures that
\begin{equation*}
B=M-1+\frac{51}{40M^2}
\end{equation*}
defines an upper bound for $W (M,B)=0$ when $M>M_{51/40}.$ Particularizing the analysis done in the
proof of this proposition to this particular $\alpha$ we can obtain an explicit value $M_{51/40}$.
It can be taken as the largest positive root of the equation
\begin{equation}\label{eq:MM}
\begin{aligned}
&196608000000{M}^{17}-1441792000000{M}^{16}-535756800000{M}^{15}\\
+&1480294400000{M}^{14}-1310515200000{M}^{13}-1151979520000{M}^{12}\\
+&1314478080000{M}^{11}-727741440000{M}^{10}-273666816000{M}^{9}\\
+&443460096000{M}^{8}-458441856000{M}^{7}+61550064000{M}^{6}+227310753600{M}^{5}\\
-&162364824000{M}^{4}-41403030120{M}^{3}+82806060240{M}^{2}-17596287801=0,
\end{aligned}
\end{equation}
which is approximately $7.58$. 

Then the result follows from Lemma~\ref{le:BM_bm}. Moreover, the polynomial that defines $\widetilde
{m}$ is obtained computing the resultant with respect to $b$ of \eqref{eq:BMtobm} and the polynomial
corresponding to \eqref{eq:MM} translated to the variables $m$ and $b$.
\end{proof}

It is useful to introduce the following notation for the lower and
upper bounds of $b^*(n)$ given in Theorem~\ref{thm:cota_global}
\begin{equation}\label{eq:bounds}
b_\ell(m):=\max\left(\frac{5m}7,m-1\right)\,\,\,\mbox{ and }\,\,\,
b_u(m):=\min\left(\dfrac{(5+\frac{37}{12} m ) m }{7+\frac{37}{12} m }, m
-1+\dfrac{25}{7 m } \right).
\end{equation}
Notice that the non-differentiability points of $b_\ell$ and $b_u$ are at $m=7/2$
and $m=7,$ respectively.

\begin{proof}[Proof of Theorem~\ref{thm:cota_global}]
Proposition~\ref{pr:cubica} provides the lower bound given in the statement, when $m\leq 7/2$. For
$m>7/2$, Perko gives a proof in~\cite{Per92}. A different one is presented in 
Lemma~\ref{lem:rectainva}.

The proof of the second part is done by comparison of the curves in the statement with
the ones provided by Proposition~\ref{prop:M_pequeno} and Lemma~\ref{le:disc}.

For $m\geq 7,$ $b=b_u(m)$ is an upper bound from
Lemma~\ref{le:disc} because $\widetilde {m}<7$ and $25/7>51/20$.

When $m\leq 7$ the proof starts translating the curve $D(M,B)=0$
given in Proposition~\ref{prop:M_pequeno} to a new algebraic curve
$E(m,b)=0$, of degree $25$ with $257$ monomials. Now we compare
the curves $(84+37m)b-(60+37m)m=0,$ corresponding to $b_u(m)=0$, and
$E(m,b)=0$ when $m,b>0.$ The resultant with respect to $b$ of
both polynomials takes the form $m^{15}p_{34}(m),$ where $p_{34}$ is a
polynomial of degree $34$ with a unique positive zero, $m_1$, as can
be easily seen from the Sturm method. Hence, the curves only
intersect at $(0,0)$ and $(m_1,b_1),$ where $m_1\approx7.1$ A
local study of the curves close to the origin shows that $b=b_u(m)$
is above $E(m,b)=0$. Hence, as the relative position of the
graphs of both curves does not change when $0<m\leq 7,$ $b=b_u(m)$
is also an upper bound of $b=b^*(m)$ in the full interval.
\end{proof}

\begin{corollary}\label{cor:cotas}
Set $\widetilde b(m)=(b_u(m)+b_\ell(m))/2$, where $b_\ell$ and $b_u$ are given in~\eqref{eq:bounds}.
Then, the absolute and relative errors when we approximate $b^*(m)$ by $\widetilde b(m)$, are 
\[
\max_{m>0} |b^*(m)-\widetilde b(m)|<\frac{37}{122}<0.31\quad\mbox{and}\quad
\max_{m>0} \left|\frac{b^*(m)-\widetilde b(m)}{b^*(m)}\right|<\frac{37}{305}<0.13.
\]
\end{corollary}

\begin{proof}
It is not difficult to see that the maxima of the functions $b_u-b_\ell$ and $(b_u-b_\ell)/b_\ell$
are both at $m=7/2.$ Then
\[
\max_{m>0} |b^*(m)-\widetilde b(m)|\le \max_{m>0}
\left|\frac{b_u(m)- b_\ell(m)}2\right|=\left|\frac{b_u(7/2)-
b_\ell(7/2)}2\right|=\frac{37}{122}
\]
and
\[
\max_{m>0} \left|\frac{b^*(m)-\widetilde b(m)}{b^*(m)}\right|\le
\max_{m>0} \left|\frac{b_u(m)-
b_\ell(m)}{2b_\ell(m)}\right|=\left|\frac{b_u(7/2)-
b_\ell(7/2)}{2b_\ell(7/2)}\right|=\frac{37}{305} .
\]
Hence, the corollary follows.
\end{proof}

\begin{remark}
From Theorem \ref{thm:cota_global}, the family of rational functions
\[
b=\frac{5+\beta m }{7+ \beta m } m ,
\]
for every $\beta>0$, approximates $b^*(m)$ in neighborhoods of the origin and the
infinity simultaneously. Nevertheless, the upper bound given in 
Theorem~\ref{thm:cota_global}, that corresponds to $\beta=37/12$, is changed in a
neighborhood of infinity because the family of functions $b= m -1+{\gamma}/{ m }$
is a much better approximation for $m$ big enough. The concrete values of $\beta$
and $\gamma$ are fixed imposing the continuity of $b_u$ and searching nice
expressions for the statement of Theorem~\ref{thm:cota_global}. These values could
be changed to obtain slightly better upper bounds.
\end{remark}

\section{Original formulation of Perko's Conjectures}\label{ap:perko}

In \cite{Per92,Per94}, Perko considers the following expression of the Bogdanov-Takens
system
\begin{equation}\label{sys:perko}
\left\{{\begin{array}{l}u'=v,\\v'=u(u-1)+\mu_1v+\mu_2uv,\end{array}}\right.
\end{equation}
and proves that it has a homoclinic saddle-loop if and only if
$\mu_1=h(\mu_2)$ for some odd analytic function $h.$ Since
$h(-\mu_2)=-h(\mu_2)$ it suffices to study $h$ either for $\mu_2>0$
or for $\mu_2<0.$

\vspace{0.3cm}

\noindent{\bf Original formulation of Perko's Conjectures (\cite{Per92}):} Let
$\mu_1=h(\mu_2)$ the function that gives the saddle-loop bifurcation curve for
system~\eqref{sys:perko}. Then,
\begin{itemize}
\item [($\mathcal{I}$)] the curve $\mu_1= h(\mu_2)$ is asymptotic to the hyperbola $\mu_1 =
-1/\mu_2$ for large $|\mu_2|$, i.e., $\mu_2h(\mu_2)+1=O(1/\mu_2)$, as $\mu_2\to\infty$,
\item [($\mathcal{II}$)] for each $\mu_2<0$, $0<h(\mu_2)< \min\{-\mu_2/7, -1/\mu_2\}. $
\end{itemize}

To see that the above conjectures are equivalent to the ones stated in the
introduction we will transform system~\eqref{sys:bt_original}
into~\eqref{sys:perko}. If we apply the change of variables $u=(x+m)/(2m)$,
$v=y/(2m)^{3/2}$ and consider the new time $s=\sqrt{2m}\,t,$
system~\eqref{sys:bt_original} writes as
\begin{equation*}
\left\{{\begin{array}{l}
\dot u=v,\\
\dot v=u(u-1)+\dfrac{b-m}{\sqrt{2m}}v+\sqrt{2m}\,uv.
\end{array}}\right.
\end{equation*}
Hence, we have the following equivalence among the parameters $m,b$
and $\mu_1,\mu_2$:
\[
\mu_1=\frac{b-m}{\sqrt{2m}},\quad \mu_2=\sqrt{2m},
\]
and every curve of the form $\mu_1=f(\mu_2)$ is transformed in the variables $m$ and
$b$ into $b=m+\sqrt{2m}\,f(\sqrt{2m})$. In particular, the curves $\mu_1=0,$
$\mu_1\mu_2=-1$ and $\mu_1=-\mu_2/7$ are transformed into the straight lines $b=m,
b=m-1$ and $b=5m/7$, respectively. This fact shows that both Conjectures $\mathcal{II}$ are
equivalent.

To compare both Conjectures $\mathcal{I}$, notice that $\mu_1\mu_2=b-m.$
Therefore, we can write $\mu_2\mu_1+1=O(1/\mu_2)$ as $b-m+1=O(1/\sqrt{m})$, as we
wanted to prove. In fact, notice that
\[ b^*(m)=m+\sqrt{2m}\,h(\sqrt{2m}).
\]
Indeed, since $h$ is analytic and odd this equality proves that
$b^*(m)$ is analytic in $m$.

It is worth to mention that there is a third conjecture in Perko's
work: {\it For $\mu_2<0$ the function $h(\mu_2)$ has a unique
maximum}. The tools introduced in this paper seem not to be adequate
to approach this question.

\end{document}